\documentclass[a4]{amsart}

\input xypic 
\input xy 
\xyoption{all} 
\usepackage{amssymb} 
\usepackage{bbm}
\usepackage{tikz-cd}

\usepackage{float}
\usepackage{graphicx}
\usepackage{accents}

\usepackage{placeins}
\usepackage[bookmarksnumbered, bookmarksopen, 
colorlinks,citecolor=blue,linkcolor=blue,backref]{hyperref}
\usepackage{enumitem}

\newtheorem{theorem}{Theorem}[section]
\newtheorem{proposition}[theorem]{Proposition}
\newtheorem{lemma}[theorem]{Lemma}
\newtheorem{corollary}[theorem]{Corollary}

\theoremstyle{definition}

\newtheorem{remark}[theorem]{Remark}
\newtheorem{example}[theorem]{Example}

\newtheorem{problem}[theorem]{Problem}

\newcounter{RomanNumber}

\newcommand{\rev}[1]{#1}

\newcounter{bean}

\newcommand{\larrow}{\relbar\!\!\relbar\!\!\rightarrow}
\newcommand{\llarrow}{\relbar\!\!\relbar\!\!\larrow}

\newcommand{\qqed}{\hfill\Box}

\begin{document}

\title
[Equivariant degrees]{On the degrees of equivariant maps from spheres to complex Stiefel manifolds}

\author{Haibao Duan}
\address{Yau Mathematical Science Center, Tsinghua University, Beijing 100084  
	\newline
	\indent
	Academy of Mathematics and Systems Science, Chinese Academy of Sciences, Beijing 100190, China}
\email{dhb@math.ac.cn}
\thanks{}

 \author{Ruizhi Huang} 
\address{State Key Laboratory of Mathematical Sciences \& Institute of Mathematics, Academy of Mathematics and Systems Science, Chinese Academy of Sciences, Beijing 100190, China
\newline
\indent School of Mathematical Sciences, University of Chinese Academy of Sciences, Beijing 100049, China} 
\email{huangrz@amss.ac.cn} 
   \urladdr{https://sites.google.com/site/hrzsea/}



\begin{abstract} 
We study the set of degrees of $\mathbb{Z}/m$-equivariant maps from spheres to complex Stiefel manifolds, motivated by the work of Astey--Gitler--Micha--Pastor. Under a suitable arithmetic condition, this set is determined using results of James, Atiyah--Todd, and Adams--Walker. Our approach is homotopy-theoretic.
\end{abstract}

\maketitle


\section{Introduction}
Let $M$ and $N$ be connected, oriented, closed manifolds of the same dimension $n$. The \emph{Brouwer degree} of a continuous map $f \colon M \larrow N$ is the integer $\mathrm{\deg }(f)$ characterized by the cohomological relation
\[
f^\ast(\omega _{N})={\rm deg}(f)\cdot \omega _{M},
\]
where $f^{\ast }$ denotes the induced homomorphism of $f$ in
cohomology, and $\omega _{N}\in H^{n}(N\rev{;\mathbb{Z}})$ and $\omega _{M}\in H^{n}(M\rev{;\mathbb{Z}})$ denote
the orientation classes of $N$ and $M$, respectively.
\rev{Throughout the paper, integral coefficients are omitted from the notation when no confusion can arise; cohomology with any other coefficient ring will be written explicitly.}

A classical and fundamental problem in algebraic and geometric topology,
dating back to the pioneering work of Brouwer~\cite{Brouwer11}, is to
determine the set of all integers that arise as mapping degrees of
continuous maps from $M$ to $N$
\begin{equation}\label{degsetdefeq}
D(M,N):=\{{\rm deg}(f)~|~f: M\larrow N\}.
\end{equation}
We refer to the references \rev{\cite{DW03,DW04,Wan02}} for historical
account and earlier results on this topic. In particular, based on C. T. C.
Wall's classification of $(n-1)$-connected $2n$-manifolds \cite{Wall62}, the
sets $D(M,N)$ were completely determined for all such manifolds $M$\ and $N$
in \cite{DW03}.

The notion of Brouwer degree admits natural generalizations and refinements,
leading to numerical invariants associated to maps between manifolds that
arise from cohomology classes not necessarily in the top dimension. For
example, let $G$ be a simple Lie group of rank $n$ with nontrivial center $%
\mathcal{Z}(G)\neq \{e\}$. Consider the quotient homomorphism
\[
p: G\larrow PG:=G/\mathcal{Z}(G),
\]
together with its induced homomorphism on the torsion-free parts of the
cohomology rings
\[
p^{\ast}: H^{\ast }(PG)/tor\cong \Lambda (y_1,\cdots, y_n)\larrow \Lambda (x_1,\cdots, x_n)\cong H^{\ast }(G)/tor,
\]
where $y_{i}$ and $x_{i}$ have the same odd degree (e.g. Chevalley 
\cite{Ch}). In \cite{DL17} and \cite[Theorem B]{DW20}, the authors
introduced the \textsl{multi-degree}\textit{\ }of $p$ as the sequence $%
\{a_{1},\cdots ,a_{n}\}$ of integers determined by
\[
p^\ast(y_i)=a_i\cdot x_i, \ \ \ 1\leq i\leq n.
\]
It is shown that the product $a_{1}\cdots a_{n}$ equals the order
of the center $\mathcal{Z}(G)$, and hence coincides with the Brouwer degree
of the covering map $p$.

This paper is concerned with continuous maps from spheres into the complex
Stiefel manifold $V_{n,k}(\mathbb{C})$, consisting of orthonormal complex $k$%
-frames in the $n$-dimensional complex vector space $\mathbb{C}^{n}$. Borel 
\cite{Borel53} showed that there exist canonical elements
\[
x_{2(n-k+1)-1},\ldots,x_{2n-1}\in H^\ast(V_{n,k}(\mathbb{C});\mathbb{Z}), \ \ \ {\rm deg}(x_j)=j, 
\]
such that the integral cohomology ring of $V_{n,k}(\mathbb{C})$ is
given by the exterior algebra
\[
H^\ast(V_{n,k}(\mathbb{C});\mathbb{Z})\cong \Lambda(x_{2(n-k+1)-1},\ldots,x_{2n-1}).
\]
To emphasize this structure, we introduce the index set
\[
I_{n,k}:=\left\{ n-k+1,n-k+2,\cdots ,n\right\} .
\]
For each integer $r\in I_{n,k}$, one can associate to a continuous
map 
\[
f: S^{2r-1}\larrow V_{n,k}(\mathbb{C})
\]
an integer $\deg (f)$, called \textsl{the degree} of $f$, via the
cohomological relation
\[
f^{\ast }(x_{2r-1})=\deg (f)\cdot \omega _{S^{2r-1}}.
\]
In analogy with \eqref{degsetdefeq}, we define the degree set
\[
D(S^{2r-1},V_{n,k}(\mathbb{C})):
=\{{\rm deg}(f)~|~f: S^{2r-1}\larrow V_{n,k}(\mathbb{C})\}, ~\text{for any $r\in I_{n,k}$},
\]
and accordingly introduce a function $h_{n,k}:I_{n,k}\larrow
\mathbb{Z}$ by letting
\[
h_{n,k}(r):= \text{the least positive integer in $D(S^{2r-1},V_{n,k}(\mathbb{C}))$}.
\]
The function $h_{n,k}$ admits an alternative interpretation via the classical Hurewicz homomorphism
\[
hur:\pi _{2r-1}(V_{n,k}(\mathbb{C}))\rightarrow H_{2r-1}(V_{n,k}(\mathbb{C})), \ \ r\in I_{n,k}.
\]
Accordingly, we refer to $h_{n,k}(r)$ as the \emph{$r$-th Hurewicz index of $V_{n,k}(\mathbb{C})$}; see Section~\ref{hnkrsec} for a precise definition and for the proof of this interpretation.

Let $\mathbb{Z}/m$ denote the cyclic group of order $m$. Both the sphere $S^{2r-1}$ and the Stiefel manifold $V_{n,k}(\mathbb{C})$ admit natural
fixed-point-free actions of $\mathbb{Z}/m$. In their study of the
triviality index of the direct sum of the canonical line bundle over a lens
space, Astey-Gitler-Micha-Pastor \cite{AGMP99} initiated the program of
classifying $\mathbb{Z}/m$-equivariant maps from $S^{2r-1}$ into $V_{n,k}(\mathbb{C})$. Motivated by this work, we consider the problem of determining
the degree set of $\mathbb{Z}/m$-equivariant maps
\[
D_{m}(S^{2r-1},V_{n,k}(\mathbb{C})):
=\{{\rm deg}(f)~|~f: S^{2r-1}\larrow V_{n,k}(\mathbb{C})~\text{is~$\mathbb{Z}/m$-equivariant}\},
\]
for any $r\in I_{n,k}$.

The main result of this paper is the following:
\begin{theorem}\label{mainthm}
For any $r\in I_{n,k}$, the degree set
\[
D(S^{2r-1},V_{n,k}(\mathbb{C}))=\left\{ a\cdot h_{n,k}(r)\mid a\in \mathbb{Z}\right\}. 
\]
Further, if $h_{n,k}(r)=1$ then the equivariant degree set $D_{m}(S^{2r-1}, V_{n,k}(\mathbb{C}))$ is either empty or  
\[
D_{m}(S^{2r-1}, V_{n,k}(\mathbb{C}))=\Big\{a\cdot m+  \binom{n}{r}~|~ a\in \mathbb{Z}\Big\}.
\]
\end{theorem}

To make Theorem~\ref{mainthm} effective for computing the equivariant degree set $D_{m}(S^{2r-1},V_{n,k}(\mathbb{C}))$, it remains to:
\begin{enumerate}[label=(\Alph*)]
\item\label{probA} determine those integers $r \in I_{n,k}$ for which $h_{n,k}(r)=1$;
\item\label{probB}  construct an equivariant map, or show that no such map exists.
\end{enumerate}

Problem \ref{probA} was studied by James \cite{James58} and Atiyah-Todd \cite{AT60}, and was essentially resolved by Adams-Walker \cite{AW65}. 
\rev{Indeed, Lemma~\ref{redhurnohlemma} reduces the computation of $h_{n,k}(r)$ to the top-dimensional index $h_{r,k-n+r}(r)$, and Lemma~\ref{h=liftlemma} identifies the condition $h_{r,k-n+r}(r)=1$ with the existence of a cross-section of the corresponding Stiefel fibration.}
Recall that every positive integer $n$ admits a unique prime decomposition
\[
n=2^{\nu_{2}(n)}\cdot 3^{\nu_{3}(n)}\cdot 5^{\nu_{5}(n)}\cdot \cdots,
\]
where the product runs over all prime numbers and $\nu_{p}(n)$
denotes the $p$-adic valuation of $n$. Thus, $n$ is uniquely determined by
the finite sequence $\left\{ \nu_{p}(n)\right\} $. In this context, the 
\textsl{Atiyah-Todd number} $M_{k}$ \cite{AT60} is defined by specifying
its $p$-adic valuations:
\begin{equation}\label{AW-eq}
 \nu_p(M_k)=\left\{\begin{array}{cc}
    {\rm max}\big(s+\nu_p(s)~|~1\leq s\leq \rev{\left\lfloor\frac{k-1}{p-1}\right\rfloor}\big) & \ \ \  \ ~{\rm if}~ p\leq k    \\
    0 & \ \ \  \ ~{\rm if}~  p> k,    \\
  \end{array}\right.
\end{equation}
where $p$ ranges over all primes. This formula provides an
effective way to compute $M_{k}$ in terms of $k$. For example, the values of 
$M_{k}$ for $1\leq k\leq 10$ are listed below.
\begin{table}[h]
	\centering
	\caption{Values of $M_k$ for $1 \leq k \leq 10$}
	\label{ATtable}
	\begin{tabular}{|c||c|c|c|c|c|c|}
		\hline
		$k$ & $1$ & $2$ & $3, \ 4$ & $5, \ 6$ & $7, \ 8$ & $9, \ 10$ \\
		\hline\hline
		$M_k$ & $1$ & $2$ & $2^3\cdot 3$ & $2^6\cdot 3^2\cdot 5$ & $2^7\cdot 3^4\cdot 5\cdot 7$  & $2^{11}\cdot 3^4\cdot 5^2\cdot 7$\\
		\hline
	\end{tabular}
\end{table}

Building on the work of Adams-Walker~\cite{AW65}, we establish the following characterization as a complement to Theorem~\ref{mainthm}.

\begin{theorem}\label{mainthmAW}
The condition $h_{n,k}(r)=1$ in Theorem~\ref{mainthm} is
equivalent to the divisibility condition\textsl{\ }$M_{r-(n-k)}\mid r$.
\end{theorem}

Regarding Problem~\ref{probB} on the existence of an equivariant map, we present two special cases.

\begin{theorem}[Theorems \ref{n-k+1degfthm} and \ref{n-k+2degfthm}]\label{mainthmex}
When $r=n-k+1$, or when $r=n-k+2$ is even and $
\binom{n}{1}\equiv \binom{n}{k-1}\equiv \binom{n}{k}\equiv 0~{\rm mod}~m
$, 
we have 
\[
D_{m}(S^{2r-1}, V_{n,k}(\mathbb{C}))=\Big\{a\cdot m+  \binom{n}{r}~|~ a\in \mathbb{Z}\Big\}.
\]
\end{theorem} 

As remarked in Example \ref{n-k+3ex}, Astey-Gitler-Micha-Pastor \cite{AGMP99} show that in the case $r=n-k+3$ there exists no $\mathbb{Z}/m$-equivariant map $S^{2(n-k+3)-1}\larrow V_{n,k}(\mathbb{C})$ under certain conditions on $n$, $k$, and $m$.

The paper is organized as follows. 
In Section \ref{hnkrsec}, we introduce the Hurewicz indices of complex Stiefel manifolds and prove Theorem \ref{mainthmAW}, based on the work of Atiyah-Todd and Adams-Walker. These numbers are nonequivariant degrees and will be used in the equivariant setting.
In Section \ref{cohomsec}, we study the integral cohomology ring of the complex projective Stiefel manifold $PV_{n,k}(\mathbb{C})$ based on the work of Astey-Gitler-Micha-Pastor. 
In Section \ref{condisec}, we deduce numerical consequences from the existence of a $\mathbb{Z}/m$-equivariant map from spheres to complex Stiefel manifolds.
Section \ref{generalsec} is devoted to proving Theorem \ref{mainthm} by studying a homotopy lifting problem. In particular, we construct new equivariant maps from a given one using obstruction theory.
We then consider two special cases in Sections \ref{case1sec} and \ref{case2sec}, respectively, and prove Theorem \ref{mainthmex}. We end this paper with an open problem.

$\, $

\noindent{\em Acknowledgements.}
\rev{The authors are grateful to the referee for careful reading and helpful comments.}

Haibao Duan was partially supported by National Natural Science Foundation of China (Grant no. 12331003).

Ruizhi Huang was supported in part by the National Natural Science Foundation of China (Grant nos. 12331003 and 12288201), the National Key R\&D Program of China (No. 2021YFA1002300).
\section{Hurewicz indices of complex Stiefel manifolds}\label{hnkrsec}
In this section, we consider the degrees of maps from spheres to complex Stiefel manifolds before proceeding to the equivariant setting in the sequel. This is equivalent to considering a numerical invariant of complex Stiefel manifolds, defined as follows. It will play an important role in the study of equivariant maps from spheres to complex Stiefel manifolds.

Let $f: S^{2r-1}\larrow V_{n,k}(\mathbb{C})$ be a map with $n-k+1\leq r\leq n$. It induces a ring homomorphism \cite{Borel53}
\[\label{cohomvnkeq}
\Lambda(x_{2(n-k+1)-1},\ldots,x_{2n-1}) \cong H^\ast(V_{n,k}(\mathbb{C});\mathbb{Z})\larrow 
H^\ast(S^{2r-1};\mathbb{Z}) \cong  \Lambda (\omega _{S^{2r-1}}),
\]
where $x_{2i-1}$ and $\omega _{S^{2r-1}}$ have degrees $2i-1$ and $2r-1$, respectively. Since the rational cohomology
\[
H^\ast(V_{n,k}(\mathbb{C});\mathbb{Q})\cong \Lambda_{\mathbb{Q}}(x_{2(n-k+1)-1},\ldots,x_{2n-1})
\]
endowed with the trivial differential is a minimal Sullivan model of $V_{n,k}(\mathbb{C})$, it follows that $V_{n,k}(\mathbb{C})$ is rationally formal, and each generator $x_{2r-1}$ is rationally spherical. Accordingly, there is a $\mathbb{Z}$-summand in the homotopy group $\pi_{2r-1}( V_{n,k}(\mathbb{C}))$ for each $n-k+1\leq r\leq n$. \rev{Choose a generator $1$ of this $\mathbb{Z}$-summand. The absolute value of the Hurewicz image of $1$ in}
\[
\mathbb{Z}\hookrightarrow \pi_{2r-1}(V_{n,k}(\mathbb{C}))\stackrel{hur}{\larrow} H_{2r-1}(V_{n,k}(\mathbb{C}))
\]
\rev{determines a positive integer} $h_{n,k}(r)$, called \textit{the $\mathit{r}$-th Hurewicz index of $V_{n,k}(\mathbb{C})$}. 
\begin{lemma}\label{degfhlemma}
Let $f: S^{2r-1}\larrow V_{n,k}(\mathbb{C})$ be a map with $n-k+1\leq r\leq n$. Then 
\[
h_{n,k}(r)~|~{\rm deg}(f). 
\]
Furthermore, there exists a map $S^{2r-1}\larrow V_{n,k}(\mathbb{C})$ whose degree is $h_{n,k}(r)$.
\end{lemma}
\begin{proof}
The naturality of Hurewicz homomorphism implies a commutative diagram 
\begin{gather*}
	\begin{aligned}
		\xymatrix{
		\pi_{2r-1}(S^{2r-1})\ar[r]^{f_\ast} \ar[d]^{hur}_{\cong}  & \pi_{2r-1}(V_{n,k}(\mathbb{C})) \ar[d]^{hur} \\
		H_{2r-1}(S^{2r-1})\ar[r]^{f_\ast}  & 
		H_{2r-1}(V_{n,k}(\mathbb{C})). 
		}
	\end{aligned}
\end{gather*}
Then ${\rm deg}(f)=\pm t\cdot h_{n,k}(r)$ where $t\in\mathbb{Z}\subseteq\pi_{2r-1}(V_{n,k}(\mathbb{C}))$ represents the component of the homotopy class of $f$ on the $\mathbb{Z}$-summand. In particular, when $t=\pm 1$, ${\rm deg}(f)= \pm h_{n,k}(r)$. \rev{Choosing the representative whose Hurewicz image has the sign compatible with the chosen cohomology generator gives a map of degree $h_{n,k}(r)$.}
\end{proof}

Note that Lemma \ref{degfhlemma} proves the first part of Theorem \ref{mainthm}. 
The following lemma reduces the computations of Hurewicz indices to the cases of top degree. 
\begin{lemma}\label{redhurnohlemma}
	Let $n-k+1\leq r\leq n$. 
\[
h_{n,k}(r)=h_{r,k-n+r} (r).
\]
\end{lemma}
\begin{proof}
\rev{If $n=1$, then $k=r=1$ and the assertion is immediate. We henceforth assume $n>1$.}
Consider the canonical fibre bundle
\[
V_{n-1,k-1}(\mathbb{C})\larrow V_{n,k}(\mathbb{C})\stackrel{q}{\larrow} S^{2n-1}.
\]
It implies that
\[
\pi_{\rev{i}}(V_{n-1,k-1}(\mathbb{C}))\cong \pi_{\rev{i}}(V_{n,k}(\mathbb{C}))
\]
for each $\rev{i}\leq 2n-3$.
Then the naturality of Hurewicz homomorphism implies the commutative diagram 
\begin{gather*}
\begin{aligned}
\xymatrix{
\pi_{2r-1}(V_{n-1,k-1}(\mathbb{C}))\ar[r]^<<<<{\cong} \ar[d]^{hur}  &\pi_{2r-1}(V_{n,k}(\mathbb{C})) \ar[d]^{hur} \\
H_{2r-1}(V_{n-1,k-1}(\mathbb{C}))\ar[r]^<<<<{\cong}     & H_{2r-1}(V_{n,k}(\mathbb{C})),  
}
\end{aligned}
\label{hurredudia}
\end{gather*}
for any $n-k+1\leq r\leq n-1$. It follows that $h_{n,k}(r)=h_{n-1, k-1}(r)$. 
\rev{Fix $r\in \{n-k+1,\ldots,n\}$. If $r=n$, then the lemma holds trivially. Otherwise, repeating the preceding argument $(n-r)$ times gives $h_{n,k}(r)=h_{r,k-n+r}(r)$.}
\end{proof}

The computation of the Hurewicz index $h_{n,k}(r)$ can be alternatively described by a lifting problem
\begin{gather}
\begin{aligned}
\xymatrix{
& V_{r,r-(n-k)}(\mathbb{C}) \ar[d]^{q}  \\
S^{2r-1} \ar[r]_{ \lambda } \ar@{-->}[ru]  &S^{2r-1},
}
\end{aligned}
\label{indexliftprodia}
\end{gather}
where the map $q$ is the complex Stiefel fibration with fibre $V_{r-1,r-(n-k)-1}(\mathbb{C})$, and $\lambda$ is a degree $\lambda$ self-map of the sphere $S^{2r-1}$. 
\begin{lemma}\label{h=liftlemma}
The Hurewicz index $h_{n,k}(r)$ is the smallest positive integer $\lambda$ that admits a lift in Diagram \eqref{indexliftprodia}.
\end{lemma}
\begin{proof}
	Consider the canonical fibre bundle
	\[
	V_{r-1,r-(n-k)-1}(\mathbb{C})\larrow V_{r,r-(n-k)}(\mathbb{C})\stackrel{q}{\larrow} S^{2r-1}.
	\]
Since the Serre spectral sequence of the bundle collapses, we have $q^\ast(\omega _{S^{2r-1}})=x_{2r-1}$ on the cohomology generators. Therefore, a lift $f:S^{2r-1}\larrow V_{r,r-(n-k)}(\mathbb{C})$ in Diagram \eqref{indexliftprodia} implies that ${\rm deg}(f)={\rm deg}(\lambda)=\lambda$. The lemma then follows from Lemmas \ref{degfhlemma} and \ref{redhurnohlemma}.
	\end{proof}

For later use, we are particularly interested in the case when $h_{n,k}(r)=1$. In terms of the lifting problem \eqref{indexliftprodia}, this case was studied by James \cite{James58} and Atiyah-Todd \cite{AT60}, and was solved by Adams-Walker \cite{AW65}. Recall that the Atiyah-Todd number $M_k$ is defined in \eqref{AW-eq}.
\begin{theorem}[\rev{Adams}-Walker \cite{AW65}]\label{awthm}
The lifting problem \eqref{indexliftprodia} has a solution for $\lambda=1$ if and only if 
$\label{sectionequicond}
M_{r-(n-k)}~|~r
$.
\end{theorem}

\begin{proof}[Proof of Theorem \ref{mainthmAW}]
By Lemma \ref{h=liftlemma}, the Hurewicz index $h_{n,k}(r)=1$ if and only if the lifting problem \eqref{indexliftprodia} has a solution for $\lambda=1$. By Theorem \ref{awthm}, the latter is equivalent to 
$
M_{r-(n-k)}~|~r
$.
\end{proof}

\begin{corollary}\label{1-3hnkrcorollary}
$h_{n,k}(r)=1$ if
\begin{itemize}
\item[(1)] $r=n-k+1$,
\item[(2)] $r=n-k+2$ is even, 
\item[(3)] or $r=n-k+3$ and $24~|~r$.
\end{itemize}
\end{corollary}
\begin{proof}
The corollary follows from Theorem \ref{mainthmAW} and Table \ref{ATtable}.
\end{proof}

\section{The integral cohomology of $PV_{n,k}(\mathbb{C})$}\label{cohomsec}

Recall the standard free action of $S^1$ on the complex Stiefel manifold $V_{n,k}(\mathbb{C})$ is defined by $z(v_1,\ldots,v_k)=(zv_1,\ldots, zv_k)$ for any complex number $z\in S^1$ and any orthonormal complex $k$-frame $(v_1,\ldots,v_k)$ in $\mathbb{C}^n$. The orbit manifold of this $S^1$-action is the complex projective Stiefel manifold $PV_{n,k}(\mathbb{C})$, and there is the principal circle bundle
\begin{equation}\label{pvnkdefeq}
S^1\larrow V_{n,k}(\mathbb{C})\stackrel{p_{n,k}}{\larrow}PV_{n,k}(\mathbb{C}).
\end{equation}
In this section, we investigate the integral cohomology ring of $PV_{n,k}(\mathbb{C})$. It will be used to prove Theorem \ref{mainthm} in the sequel.

We begin with some basic facts summarized in \cite[Section~2]{AGMP99} by Astey-Gitler-Micha-Pastor. 
Let $\mathbb{U}(n)$ be the $n$-th unitary group. 
There is a homotopy pullback diagram
\begin{gather}
\begin{aligned}
\xymatrix{
V_{n,k}(\mathbb{C}) \ar@{=}[r]  \ar[d]^{p_{n,k}}  & V_{n,k}(\mathbb{C}) \ar[d] \\
PV_{n,k}(\mathbb{C}) \ar[r]^{f} \ar[d]^{\omega}   & B\mathbb{U}(n-k)\ar[d]^{Bi}\\
\mathbb{CP}^{\infty} \ar[r]^{f_0}   & B\mathbb{U}(n),
}
\end{aligned}
\label{pvnkdefdiag}
\end{gather}
where the fibre bundles in the left and right columns are induced from (\ref{pvnkdefeq}) and the canonical principal bundle $\mathbb{U}(n-k)\stackrel{i}{\larrow}\mathbb{U}(n)\stackrel{}{\larrow}V_{n,k}(\mathbb{C})$ respectively, the map $f_0$ classifies the sum of $n$ copies of the Hopf line bundle over $\mathbb{C}P^{\infty}$, and the map $f$ classifies the bundle whose fibre over a point $\nu$ in $PV_{n,k}(\mathbb{C})$ is the orthogonal complement in $\mathbb{C}^n$ of the subspace generated by a $k$-frame representing $\nu$.

For the Serre spectral sequence of the bundle in the right column of Diagram \eqref{pvnkdefdiag}, it is known that
the transgression
\[
\tau(x_{2r-1})=c_r \in H^{2r}(B\mathbb{U}(n);\mathbb{Z}),
\]
where $c_r$ is the $r$-th universal Chern class. 
By the naturality of transgression, this implies that, for the Serre spectral sequence of the bundle in the left column of Diagram \eqref{pvnkdefdiag}, 
\begin{equation}\label{tran-w-eq}
\tau(x_{2r-1})={n\choose r} u^r,
\end{equation}
where the generator $u\in H^2(\mathbb{C}P^{\infty};\mathbb{Z})$ is the first Chern class of the Hopf line bundle. The following lemma describes the rational cohomology ring of $PV_{n,k}(\mathbb{C})$.
\begin{lemma}\label{qcohompvnklemma}
There is a graded ring isomorphism
\[
H^{\ast}(PV_{n,k}(\mathbb{C});\mathbb{Q})\cong \Lambda(x_{2(n-k+2)-1}, \ldots, x_{2n-1})\otimes \mathbb{Q}[\omega]/(\omega^{n-k+1}),
\]
where for simplicity of notation we denote by $\omega=\omega^\ast(u)\in H^2(PV_{n,k}(\mathbb{C});\mathbb{Z})$ the Euler class of (\ref{pvnkdefeq}). ~$\qqed$
\end{lemma}

Since we are interested in the mapping degree, in addition to the rational cohomology we also need sufficient information about the integral cohomology. \rev{We include the proof below because an earlier statement of the integral cohomology ring of complex projective Stiefel manifolds in \cite{Ruiz69} is not correct in the form needed here.} For any triple $(n,k,r)$ of positive integers such that $n-k+2\leq r\leq n$, let us define
\begin{equation}\label{abeq}
a_{n,k,r}=\frac{b_{n,k,r-1}}{b_{n,k,r}},
\end{equation}
where 
\begin{equation}\label{bdefeq}
b_{n,k,r}={\rm g.c.d.}~\Big\{\binom{n}{n-k+1},\cdots, \binom{n}{r} \Big\}.
\end{equation}
Note that $b_{n,k,r}~|~b_{n,k,r-1}$, $b_{n,k,r}~|~b_{n,k-1,r}$, $b_{n,k,r}={\rm g.c.d.}(b_{n,k,r-1}, b_{n,k-1,r})$ and $b_{n, k, n}=1$.
\begin{theorem}\label{cohompvnkthm}
There is a graded ring epimorphism
\[
\Lambda(\bar{x}_{2r-1}~|~n-k+2\leq r\leq n) \otimes J(\omega)\twoheadrightarrow H^\ast(PV_{n,k}(\mathbb{C});\mathbb{Z}),
\]
such that
\begin{itemize}
\item[(1).] $J(\omega):=\mathbb{Z}[\omega]/(b_{n,k,r}\omega^r~|~n-k+1\leq r\leq n)$,
\item[(2).] The image of $\omega^\ast: H^\ast(\mathbb{C}P^{\infty};\mathbb{Z})\larrow H^\ast(PV_{n,k}(\mathbb{C});\mathbb{Z})$ is exactly $J(\omega)$,
\item[(3).] $\Lambda(\bar{x}_{2r-1}~|~n-k+2\leq r\leq n)$ is a subalgebra of $H^\ast(PV_{n,k}(\mathbb{C});\mathbb{Z})$,
\item[(4).] $p_{n,k}^\ast(\bar{x}_{2r-1})= a_{n,k,r} x_{2r-1}\in H^{2r-1}(V_{n,k}(\mathbb{C});\mathbb{Z})$.
\end{itemize}
\end{theorem}

\begin{proof}
By the theory of Serre spectral sequence (for instance, see \cite{McCleary01}), there is the following commutative diagram for the Serre spectral sequence $(E_{\ast}^{\ast,\ast}, d_\ast)$ of the fibre bundle in the left column of Diagram \eqref{pvnkdefdiag}
\begin{gather}
\begin{aligned}
\xymatrix{
H^{2r-1}(PV_{n,k}(\mathbb{C})) \ar[r]^{p_{n,k}^\ast}  \ar@{->>}[d]^{p_{\infty}}    & H^{2r-1}(V_{n,k}(\mathbb{C})) \ar[r]^{\delta \ \ \ \ \ \ }    & H^{2r}(PV_{n,k}(\mathbb{C}), V_{n,k}(\mathbb{C}))\ar[r]^{ \ \ \ \ \ \   j^\ast}   & H^{2r}(PV_{n,k}(\mathbb{C})) \\
E_{\infty}^{0,2r-1}\cong E_{2r+1}^{0,2r-1} \ar@{^{(}->}[r]^{\ \ \ \ \ i_{\infty, 2r}}    &  E_{2r}^{0,2r-1}\ar[r]^{d_{2r}}  \ar@{^{(}->}[u]^{i_{2r,2}}  & E_{2r}^{2r,0}   \ar@{^{(}->}[u] 
\ar@{->>}[r]  &  E_{2r+1}^{2r,0}\cong  E_{\infty}^{2r,0}  \ar@{^{(}->}[u]\\
&& H^{2r}(\mathbb{C}P^{\infty},\ast) \ar@{->>}[u]  \ar@/_2.7pc/[uu]_>(.8){ \omega^\ast} \ar[r]^{j^\ast}_{\cong}  & H^{2r}(\mathbb{C}P^{\infty}).  \ar@{->>}[u]   \ar@/_2.7pc/[uu]_>(.8){ \omega^\ast}
}
\end{aligned}
\label{transdiag}
\end{gather}
Likewise, there is a similar commutative diagram for the Serre spectral sequence $(\widetilde{E}_{\ast}^{\ast,\ast}, \widetilde{d}_\ast)$ of the bundle in the right column of Diagram \eqref{pvnkdefdiag}, and it maps to the above diagram by the naturality of Serre spectral sequence. Accordingly, we have 
\begin{gather*}
\begin{aligned}
\xymatrix{
H^{2r-1}(V_{n,k}(\mathbb{C}))   \ar@{=}[r]  & H^{2r-1}(V_{n,k}(\mathbb{C})) \\
\widetilde{E}_{2r}^{0,2r-1}  \ar@{^{(}->}[u]^{\cong}  \ar[r]       & E_{2r}^{0,2r-1}  \ar@{^{(}->}[u]^{i_{2r,2}},
}
\end{aligned} 
\end{gather*}
where the left inclusion is an isomorphism by inspecting the spectral sequence. As $i_{2r,2}$ is injective, it follows that the other two morphisms in the above square are isomorphisms.

By Lemma \ref{qcohompvnklemma}, there exists a set of elements $\{\bar{x}_{2r-1}  ~|~ n-k+2\leq r\leq n\}$ (${\rm deg}(x_{2r-1})=2r-1$) such that $\Lambda(\bar{x}_{2r-1}~|~n-k+2\leq r\leq n)$ is a subalgebra of $H^\ast(PV_{n,k}(\mathbb{C});\mathbb{Z})$ and each $\bar{x}_{2r-1}$ is a cohomology generator. In particular, $p_{\infty}$ in Diagram \eqref{transdiag} maps the free summand $\mathbb{Z}\{\bar{x}_{2r-1}\}$ isomorphically to $E_{\infty}^{0,2r-1}$. 
As we have shown that $i_{2r,2}$ is an isomorphism, we can identify $p_{n,k}^\ast$ with $i_{\infty, 2r}$ on this $\mathbb{Z}$-summand by the left square of Diagram \eqref{transdiag}. Furthermore, in Diagram \eqref{transdiag} $d_{2r}$ is the transgression and then by \eqref{tran-w-eq}
\[
d_{2r}(x_{2r-1})={n\choose r} \omega^r.
\]

Combining the above, \rev{the maps in the second row of Diagram \eqref{transdiag} may be restricted to the relevant free summands and cyclic subgroups generated by $x_{2r-1}$ and $\omega^r$.} We can inductively show that there is an exact subsequence for the summand $\mathbb{Z}\{\bar{x}_{2r-1}\}$ ($n-k+2\leq r\leq n$)
\[
0\larrow \mathbb{Z}\{\bar{x}_{2r-1}\} \stackrel{ \times  a_{n,k,r}}{\llarrow}\mathbb{Z}\{x_{2r-1}\}\stackrel{\times {n\choose r}}{\llarrow} \mathbb{Z}/b_{n,k,r-1}\{\omega^r\} \llarrow \mathbb{Z}/b_{n,k,r}  \{\omega^r\}  \larrow 0, 
\]
\rev{Here the second map is multiplication by $a_{n,k,r}$ because $p_{n,k}^\ast(\bar{x}_{2r-1})=a_{n,k,r}x_{2r-1}$ on the free summand. The group immediately before $d_{2r}$ is $\mathbb{Z}/b_{n,k,r-1}\{\omega^r\}$ by the previous differentials, while quotienting further by $d_{2r}(x_{2r-1})=\binom{n}{r}\omega^r$ gives $\mathbb{Z}/b_{n,k,r}\{\omega^r\}$, since $b_{n,k,r}=\gcd(b_{n,k,r-1},\binom{n}{r})$.}
and in particular,
\[
p_{n,k}^\ast(\bar{x}_{2r-1})= a_{n,k,r} x_{2r-1}
~\ \ \text{
and} ~ \ \ 
0=b_{n,k,r}\omega^r\in E_{\infty}^{2r,0}\subseteq H^{2r}(PV_{n,k}(\mathbb{C})).
\]
The latter implies that the image of the edge homomorphism $\omega^\ast$ in the last column of Diagram \eqref{transdiag} is $J(\omega)$.

We have proved all the statements in the theorem except that $\{\bar{x}_{2r-1}~|~n-k+2\leq r\leq n\}\cup \{\omega\}$ is a set of algebraic generators of $H^\ast(PV_{n,k}(\mathbb{C}))$. Suppose for contradiction there is another generator $x$. Then by Lemma \ref{qcohompvnklemma} $x$ is an indecomposable torsion element, say of order $p^i$. It implies that there exists an indecomposable element other than $x_{2r-1}$ and $\omega$ in the cohomology ring $H^\ast(PV_{n,k}(\mathbb{C});\mathbb{Z}/p)$. However, this is impossible by \cite[Theorems 1.1 and 1.2]{AGMP99} in the case $n>k$ and by \cite{Duan20} in the case $n=k$. Therefore, $\{\bar{x}_{2r-1}~|~n-k+2\leq r\leq n\}\cup \{\omega\}$ is a set of algebraic generators of $H^\ast(PV_{n,k}(\mathbb{C}))$, and the theorem is proved. 
\end{proof}

\begin{remark}
Following \cite{DW20}, we may call the sequence of numbers $\{a_{n,k,r}~|~n-k+2\leq r \leq n\}$ the {\it multi-degree} of $p_{n,k}$. 

Theorem \ref{cohompvnkthm} generalizes the corresponding result for $PV_{n,n}(\mathbb{C})=P\mathbb{U}(n)$ in \cite{DL17}. Additionally, the first author \cite{Duan20} has computed the integral cohomology ring of $P\mathbb{U}(n)$.
\end{remark}

\section{Equivariant maps from spheres to complex Stiefel manifolds}\label{condisec}
In this section, we deduce numerical consequences of a given equivariant map from spheres to complex Stiefel manifolds. The main result is Lemma \ref{neccondleema3}.

Let $L^{2r-1}(m)$ denote the orbit lens space of the sphere $S^{2r-1}$ by the standard action of the finite group $\mathbb{Z}/m$ of the $m$-roots of unity. From the principal bundle 
\[
\mathbb{Z}/m\larrow S^{2r-1} \stackrel{\rho}{\larrow} L^{2r-1}(m),
\]
there is the circle bundle of its associated complex line bundle 
\[
S^1\larrow S^{2r}(m) \stackrel{p}{\larrow} L^{2r-1}(m),
\]
where $S^{2r}(m)=S^{2r-1}\times_{\mathbb{Z}/m} S^1$.
\begin{lemma}\label{klemma}
There is a principal bundle diagram
\begin{gather*}
\begin{aligned}
\xymatrix{
\mathbb{Z}/m \ar[r] \ar@{^{(}->}[d]   &  S^{2r-1} \ar[r]^<<<<{\rho} \ar[d]^{\kappa}  &  L^{2r-1}(m) \ar@{=}[d]\\
S^1 \ar[r] &  S^{2r}(m)  \ar[r]^<<<<{p}  & L^{2r-1}(m) 
}
\end{aligned}
\end{gather*}
\rev{such that $\kappa^\ast: H^{2r-1}(S^{2r}(m);\mathbb{Z})\to H^{2r-1}(S^{2r-1};\mathbb{Z})$ is an isomorphism.}
\end{lemma}
\begin{proof}
Denote by $i: \mathbb{Z}/m\hookrightarrow S^1$ the inclusion of the subgroup. The map $1\times i:  S^{2r-1}\times \mathbb{Z}/m \larrow S^{2r-1}\times S^1$ is $\mathbb{Z}/m$-equivariant with respect to the diagonal action. It induces a map of orbit spaces 
\[
\kappa: S^{2r-1}\cong (S^{2r-1}\times \mathbb{Z}/m)/(\mathbb{Z}/m) \stackrel{}{\larrow} (S^{2r-1}\times S^1)/(\mathbb{Z}/m)=S^{2r}(m).
\]
Then there is a principal bundle diagram
\begin{gather*}
\begin{aligned}
\xymatrix{
\mathbb{Z}/m \ar[r] \ar@{^{(}->}[d]   &  (S^{2r-1}\times \mathbb{Z}/m)/(\mathbb{Z}/m)   \ar[r]^<<<<{} \ar[d]^{\kappa}  &  S^{2r-1}/(\mathbb{Z}/m) \ar@{=}[d]\\
S^1 \ar[r] &  (S^{2r-1}\times S^1)/(\mathbb{Z}/m) \ar[r]^<<<<{}  & S^{2r-1}/(\mathbb{Z}/m), 
}
\end{aligned}
\end{gather*}
which is the required principal bundle diagram in the lemma. \rev{The map $1\times i$ descends to $\kappa$ after passing to the orbit spaces; equivalently, the quotient projections form a covering diagram over $\kappa$ with group $\mathbb{Z}/m$. Since}
$(1\times i)^\ast: H^{2r-1}(S^{2r-1}\times S^1)\larrow H^{2r-1}(S^{2r-1}\times \mathbb{Z}/m)$ is an isomorphism on the direct summand $H^{2r-1}(S^{2r-1})$, it follows that \rev{$\kappa^\ast$ is an isomorphism in degree $2r-1$}.
\end{proof}

\rev{In the sequel, for $n-k+1\leq r\leq n$, we choose $\hat{\omega}_{2r-1}\in H^{2r-1}(S^{2r}(m);\mathbb{Z})$ by the condition $\kappa^\ast(\hat{\omega}_{2r-1})=\omega_{S^{2r-1}}$. For a map $g:S^{2r}(m)\to V_{n,k}(\mathbb{C})$, we define its degree in degree $2r-1$ by $g^\ast(x_{2r-1})={\rm deg}(g)\hat{\omega}_{2r-1}$. This is a cohomological degree with respect to the chosen generator, not a Brouwer degree of maps between manifolds of the same dimension. With this convention, ${\rm deg}(g\circ\kappa)={\rm deg}(g)$.}

Suppose there is a $\mathbb{Z}/m$-equivariant map
\[
f: S^{2r-1}\larrow V_{n,k}(\mathbb{C}),
\]
where $\mathbb{Z}/m$ acts on $V_{n,k}(\mathbb{C})$ as a subgroup of $S^1$. 

\begin{lemma}\label{fslemma}
The $\mathbb{Z}/m$-equivariant map
$
f: S^{2r-1}\larrow V_{n,k}(\mathbb{C})
$
factors through $\kappa$ as 
\[
f: S^{2r-1}\stackrel{\kappa}{\larrow} S^{2r}(m) \stackrel{f^s}{\larrow}  V_{n,k}(\mathbb{C}),
\]
where $f^s$ is an $S^1$-equivariant map with ${\rm deg}(f)={\rm deg}(f^s)$. Accordingly, there is a principal bundle diagram
\begin{gather*}
\begin{aligned}
\xymatrix{
S^1 \ar@{=}[d]   \ar[r] &  S^{2r}(m)  \ar[d]^{f^{s}} \ar[r]^<<<<{p}  & L^{2r-1}(m) \ar[d]^{\bar{f}}\\
S^1              \ar[r]   &  V_{n,k}(\mathbb{C}) \ar[r]^<<<<{p_{n,k}}   &   PV_{n,k}(\mathbb{C}),
}
\end{aligned}
\end{gather*}
where $\bar{f}$ is the induced map.
\end{lemma}
\begin{proof}
Consider the fibre bundle diagram 
\begin{gather*}
\begin{aligned}
\xymatrix{
S^{2r-1} \ar[r]^{f} \ar[d]^{\rho}  & V_{n,k}(\mathbb{C}) \ar@{=}[r] \ar[d]  &  V_{n,k}(\mathbb{C}) \ar[d]^{p_{n,k}} \\
L^{2r-1}(m) \ar[r]  \ar[d] &  V_{n,k}(\mathbb{C})/(\mathbb{Z}/m) \ar[r] \ar[d]  &  PV_{n,k}(\mathbb{C})  \ar[d]^{\omega} \\
B\mathbb{Z}/m \ar@{=}[r]  &  B\mathbb{Z}/m  \ar[r]^{Bi}   & BS^1,
}
\end{aligned}
\end{gather*}
where the two vertical maps to $B\mathbb{Z}/m$ classify the $m$-covers of $L^{2r-1}(m)$ and $V_{n,k}(\mathbb{C})/(\mathbb{Z}/m)$ respectively, and $\omega$ classifies the circle bundle \eqref{pvnkdefeq}. Denote by $\bar{f}$ the middle row composite. Then the circle bundle over $L^{2r-1}(m)$ classified by the composite
$
L^{2r-1}(m)\stackrel{\bar{f}}{\larrow} PV_{n,k}(\mathbb{C}) \stackrel{\omega}{\larrow}  \mathbb{C}P^{\infty}
$ 
is isomorphic to $S^1\larrow S^{2r}(m) \stackrel{p}{\larrow} L^{2r-1}(m)$. \rev{Indeed, the pullback of $V_{n,k}(\mathbb{C})/(\mathbb{Z}/m)\to PV_{n,k}(\mathbb{C})$ along $\bar f$ is the principal $\mathbb{Z}/m$-bundle $S^{2r-1}\to L^{2r-1}(m)$; extending its structure group by $i:\mathbb{Z}/m\hookrightarrow S^1$ gives precisely $S^{2r-1}\times_{\mathbb{Z}/m}S^1=S^{2r}(m)$.} Hence there is a fibre bundle diagram 
\begin{gather*}
\begin{aligned}
\xymatrix{
S^{2r-1} \ar[r]^{\kappa} \ar[d]^{\rho}  &  S^{2r}(m) \ar[r]^{f^s} \ar[d]^{p}  &  V_{n,k}(\mathbb{C}) \ar[d]^{p_{n,k}} \\
L^{2r-1}(m) \ar@{=}[r]  \ar[d] & L^{2r-1}(m)  \ar[r]^{\bar{f}} \ar[d]^{\omega}  &  PV_{n,k}(\mathbb{C})  \ar[d]^{\omega} \\
B\mathbb{Z}/m \ar[r]^{Bi}  &  BS^1 \ar@{=}[r]  & BS^1,
}
\end{aligned}
\end{gather*}
where $f^s$ is the induced map, the left part follows from Lemma \ref{klemma} and the right part implies the bundle diagram in the lemma. Comparing the above two bundle diagrams, we see that $f$ factors as $f^s\circ \kappa$ with $f^s$ an $S^1$-equivariant map. \rev{By the choice of $\hat{\omega}_{2r-1}$ above, this factorization implies ${\rm deg}(f)={\rm deg}(f^s)$.}
\end{proof}

\rev{Let $\Omega_L\in H^{2r-1}(L^{2r-1}(m);\mathbb{Z})$ be the orientation generator satisfying $\rho^\ast(\Omega_L)=m\omega_{S^{2r-1}}$. Then the equality $\kappa^\ast p^\ast(\Omega_L)=\rho^\ast(\Omega_L)=m\omega_{S^{2r-1}}$ and Lemma~\ref{klemma} imply $p^\ast(\Omega_L)=m\hat{\omega}_{2r-1}$. The labels $\rho^\ast=\times m$ and $p^\ast=\times m$ in the following diagram refer to these generators. Moreover, for $n-k+2\leq r\leq n$ and for a map $\phi:L^{2r-1}(m)\to PV_{n,k}(\mathbb{C})$, we define ${\rm deg}(\phi)$ by $\phi^\ast(\bar{x}_{2r-1})={\rm deg}(\phi)\Omega_L$ whenever this notation is used below.}

Applying the Gysin sequence to the circle bundles in Lemma \ref{fslemma}, we have the commutative diagram
\begin{gather}
\begin{aligned}
\xymatrix{
  & \mathbb{Z}\cong H^{2r-1}(L^{2r-1}(m))  \ar@{=}[d]  \ar[r]^<<<<{\rho^\ast=\times m} &
  H^{2r-1}(S^{2r-1})\cong \mathbb{Z} \\
  0 \ar[r]^<<<<<{\cup \omega}     &
  H^{2r-1}(L^{2r-1}(m))             \ar[r]^{p^\ast=\times m}  &
  H^{2r-1}(S^{2r}(m)) \ar[r]^<<<<{\Theta}  \ar[u]^{\kappa^\ast}_{\cong} &
  H^{2r-2}(L^{2r-1}(m))\cong \mathbb{Z}/m\{\omega^{r-1}\} \ar[r]^<<<{\cup \omega}&
  0\\
 &H^{2r-1}(PV_{n,k}(\mathbb{C})) \ar[r]^{p_{n,k}^\ast}  \ar[u]^{\bar{f}^\ast}&
 H^{2r-1}(V_{n,k}(\mathbb{C})) \ar[r]^<<<<<<<{\Theta} \ar[u]^{f^{s\ast}} &
 H^{2r-2}(PV_{n,k}(\mathbb{C}))  \ar[u]^{\bar{f}^\ast},
}
\end{aligned}
\label{gysinequifdia}
\end{gather}
where the upper left square is induced from the right square of the diagram in Lemma \ref{klemma}, 
the second and third rows are exact, the maps $\Theta$ are the connecting homomorphisms, and $\omega$ is the Euler class by abuse of notation. 
\begin{lemma}[Borel \cite{Borel53}]\label{Borellemma}
For $n-k+1\leq r\leq n$, the connecting homomorphism $\Theta: H^{2r-1}(V_{n,k}(\mathbb{C})) \to H^{2r-2}(PV_{n,k}(\mathbb{C}))$ satisfies
\[
\hspace{5.8cm}
\Theta(x_{2r-1})=\binom{n}{r}\omega^{r-1}. 
\hspace{5.8cm}
\]
\end{lemma}
\begin{proof}
\rev{This is the Gysin-sequence form of Borel's transgression formula. More explicitly, the Serre transgression for the left column of Diagram~\eqref{pvnkdefdiag} is $\tau(x_{2r-1})=\binom{n}{r}\omega^r$ by \eqref{tran-w-eq}. Under the standard identification between this transgression and cup product with the Euler class after the Gysin connecting homomorphism, $\omega\Theta(x_{2r-1})=\binom{n}{r}\omega^r$, which gives the stated formula.}
\end{proof}

The following lemma is key to studying the degrees of equivariant maps from spheres to complex Stiefel manifolds.
\begin{lemma}\label{neccondleema3}
For $n-k+1\leq r\leq n$, 
\[
{\rm deg}(f)\equiv \binom{n}{r}~{\rm mod}~m.
\]
\end{lemma}
\begin{proof}
\rev{With the generator $\hat{\omega}_{2r-1}\in H^{2r-1}(S^{2r}(m))$ chosen above, the second row of Diagram \eqref{gysinequifdia} gives $\Theta(\hat{\omega}_{2r-1})=\omega^{r-1}$.} We have
\[
\begin{split}
0&=\Theta f^{s\ast}(x_{2r-1})-\bar{f}^\ast\Theta(x_{2r-1})\\
&={\rm deg}(f^s)\Theta(\rev{\hat{\omega}_{2r-1}})-  \bar{f}^\ast\Big(\binom{n}{r} \omega^{r-1}\Big)\\
&=\Big({\rm deg}(f)-\binom{n}{r} \Big)\omega^{r-1} \in \mathbb{Z}/m\{\omega^{r-1}\},
\end{split}
\]
where the first equality follows from the lower right square of Diagram \eqref{gysinequifdia}, and the second and third equalities follow from Lemmas \ref{Borellemma} and \ref{fslemma} respectively. Then the congruence follows.
\end{proof}

Although they are not used in this paper, the following two lemmas are of independent interest and may be useful in the further study of equivariant maps from spheres to complex Stiefel manifolds.

\begin{lemma}\label{neccondleema1}
	For $n-k+2\leq r\leq n$, 
	\[
	m ~{\rm deg}(\bar{f})={\rm deg}(f)a_{n,k,r}.
	\]
\rev{Here ${\rm deg}(\bar f)$ is understood in the sense of the preceding convention.}
\end{lemma}
\begin{proof}
	By Theorem \ref{cohompvnkthm}, the lower left square of Diagram \eqref{gysinequifdia} implies that 
	\[
	(p^\ast\circ \bar{f}^\ast)(\bar{x}_{2r-1})=(f^{s\ast}\circ p_{n,k}^\ast)(\bar{x}_{2r-1})=f^{s\ast}(a_{n,k,r} x_{2r-1})=a_{n,k,r}f^{s\ast}(x_{2r-1})
	\]
	Since ${\rm deg}(f)={\rm deg}(f^s)$ \rev{from the proof of Lemma \ref{fslemma}} and $p^\ast$ is multiplication by $m$, we see that $m ~{\rm deg}(\bar{f})={\rm deg}(f)a_{n,k,r}$.
\end{proof}

\begin{lemma}\label{neccondleema2}
	For $n-k+1\leq r\leq n$, 
	\[
	m~|~b_{n,k,r}.
	\]
\end{lemma}
\begin{proof}
	Consider $\bar{f}^\ast: H^\ast(PV_{n,k}(\mathbb{C}))\larrow H^{2r-2}(L^{2r-1}(m))\cong \mathbb{Z}/m\{\omega^{r-1}\}$ in Diagram \eqref{gysinequifdia}. Since $\bar{f}^\ast(\omega)=\omega$ by Lemma \ref{fslemma}, the order of $\omega^{r-1}\in H^{2r-2}(PV_{n,k}(\mathbb{C}))$ is divisible by $m$, where this order is $b_{n,k,r-1}$ by Theorem \ref{cohompvnkthm}.
\end{proof}


\section{A lifting problem}\label{generalsec}
In this section, we present a general method for constructing equivariant degrees from a given one, and then combine Lemma \ref{neccondleema3} to prove Theorem \ref{mainthm}.

Consider the lifting problem
\begin{gather}
\begin{aligned}
\xymatrix{
& PV_{n,k}(\mathbb{C}) \ar[d]^{\omega}  \\
L^{2r-1}(m) \ar[r]^<<<<<{\omega }  \ar@{-->}[ru]^{g}   
&  \mathbb{C}P^{\infty}
}
\end{aligned}
\label{rliftprodia}
\end{gather}
with $r>1$, 
where the maps $\omega$ classify the corresponding complex line bundles.
A solution $g$ induces a $\mathbb{Z}/m$-equivariant map
\begin{equation}\label{assuequigeq}
\tilde{g}: S^{2r-1}\stackrel{\kappa}{\larrow} S^{2r}(m)\stackrel{}{\larrow} V_{n,k}(\mathbb{C}),
\end{equation}
where $\kappa$ is defined in Lemma \ref{klemma} and the second map between the total manifolds of the circle bundles is induced from Diagram \eqref{rliftprodia}. 

We may modify a lifting $g$ in Diagram \eqref{rliftprodia} to obtain other liftings. To this end, let us first consider the relation between two liftings. If
\[
\hat{g}: L^{2r-1}(m) \larrow PV_{n,k}(\mathbb{C})
\]
is another lifting of $\omega $ in Diagram \eqref{rliftprodia}, then the difference $\hat{g}\circ\rho-g\circ \rho$, with $S^{2r-1}\stackrel{\rho}{\larrow} L^{2r-1}(m)$ the $\mathbb{Z}/m$-cover, factors through $V_{n,k}(\mathbb{C})$
\[
\hat{g}\circ\rho-g\circ\rho:  S^{2r-1} \stackrel{\delta}{\larrow}  V_{n,k}(\mathbb{C}) \stackrel{p_{n,k}}{\larrow}  PV_{n,k}(\mathbb{C}),
\]
for some map $\delta$.
As before $\hat{g}$ also induces a $\mathbb{Z}/m$-equivariant map
\[
\tilde{\hat{g}}: S^{2r-1}\stackrel{\kappa}{\larrow} S^{2r}(m)\stackrel{}{\larrow} V_{n,k}(\mathbb{C}).
\]

\begin{lemma}\label{rmodifydegreelemma}
The difference $\tilde{\hat{g}}-\tilde{g}= \delta$, and in particular, 
\[
{\rm deg}(\tilde{\hat{g}})-{\rm deg}(\tilde{g})={\rm deg}(\delta).
\]
\end{lemma}
\begin{proof}
Note that (cf. the diagrams in Lemmas \ref{klemma} and \ref{fslemma})
\[
p_{n,k}\circ (\tilde{\hat{g}}-\tilde{g})=\hat{g}\circ\rho-g\circ\rho=p_{n,k}\circ \delta.
\]
Then since 
$
p_{n,k\ast}: \pi_{2r-1}(V_{n,k}(\mathbb{C}))\larrow \pi_{2r-1}(PV_{n,k}(\mathbb{C}))
$
is an isomorphism unless $n=k$ and $r=1$, we have
$
\tilde{\hat{g}}-\tilde{g}= \delta
$.
\end{proof}

Recall $h_{n,k}(r)$ is the $\mathit{r}$-th Hurewicz index of $V_{n,k}(\mathbb{C})$ defined in Section \ref{hnkrsec}.  
\begin{proposition}\label{rdegbarflemma}
Let $n-k+1\leq r\leq n$. 
Suppose $\tilde{g}$ (\ref{assuequigeq}) is a $\mathbb{Z}/m$-equivariant map induced from a solution $g$ of the lifting problem \eqref{rliftprodia}.
Then for any integer $a$ with
\[
a \equiv {\rm deg}(\tilde{g})~{\rm mod}~m\cdot h_{n,k}(r),
\]
there exists a $\mathbb{Z}/m$-equivariant map $f:S^{2r-1}\larrow V_{n,k}(\mathbb{C})$ such that 
\[
{\rm deg}(f)=a.
\]
\end{proposition} 
\begin{proof}
First, since the covering map $S^{2r-1}\stackrel{\rho}{\larrow} L^{2r-1}(m)$ is of degree $m$ at the top cell, the composite 
\[
\mathbb{Z}\hookrightarrow \pi_{2r-1}(V_{n,k}(\mathbb{C})) \stackrel{\rev{c}^\ast}{\larrow}  [ L^{2r-1}(m),V_{n,k}(\mathbb{C})]  \stackrel{\rho^\ast}{\larrow}\pi_{2r-1}(V_{n,k}(\mathbb{C})) \stackrel{\rev{\operatorname{pr}}}{\larrow}\mathbb{Z}
\]
is also of degree $m$, where \rev{$c:L^{2r-1}(m)\to S^{2r-1}$ is the pinch map to the top cell and $\operatorname{pr}$ is the canonical projection to the $\mathbb{Z}$-summand}. Hence, \rev{$c^\ast$} is injective, and for the map $\epsilon: S^{2r-1}\larrow V_{n,k}(\mathbb{C})$ representing an integer $t\in \mathbb{Z}$
\[
m \epsilon=\epsilon \circ \rev{c}\circ \rho ~ \ \  {\rm and } \ \ ~{\rm deg}(\epsilon)=t\cdot  h_{n,k}(r).
\]
Second, for the circle bundle (\ref{pvnkdefeq}) over $PV_{n,k}(\mathbb{C})$ we have the exact sequence
\[
0=H^1(L^{2r-1}(m))\larrow [ L^{2r-1}(m),V_{n,k}(\mathbb{C})]\stackrel{p_{n,k\ast}}{\larrow}  [ L^{2r-1}(m),PV_{n,k}(\mathbb{C})] \larrow H^{2}(L^{2r-1}(m))\cong\mathbb{Z}/m.
\]
\rev{Combining} the above, we obtain the \rev{commutative} diagram
\begin{gather*}
\begin{aligned}
\xymatrix{
\mathbb{Z}\ar@{^{(}->}[r]&
\pi_{2r-1}(V_{n,k}(\mathbb{C})) \ar@{^{(}->}[r]^<<<<<{\rev{c}^\ast} \ar[d]^{p_{n,k\ast}}_{\cong}&
 [ L^{2r-1}(m),V_{n,k}(\mathbb{C})]  \ar[r]^<<<<<{\rho^\ast}  \ar@{^{(}->}[d]^{p_{n,k\ast}}&
 \pi_{2r-1}(V_{n,k}(\mathbb{C}))   \ar[d]^{p_{n,k\ast}}_{\cong}\\
& \pi_{2r-1}(PV_{n,k}(\mathbb{C})) \ar[r]^<<<<{\rev{c}^\ast}&
 [ L^{2r-1}(m),PV_{n,k}(\mathbb{C})]  \ar[r]^<<<<{\rho^\ast}&
 \pi_{2r-1}(PV_{n,k}(\mathbb{C})).
}
\end{aligned}
\label{rsuffdia1}
\end{gather*}
By the standard coaction $L^{2r-1}(m)\larrow L^{2r-1}(m)\vee S^{2r-1}$, the subgroup $\mathbb{Z}$ acts on $[L^{2r-1}(m),PV_{n,k}(\mathbb{C})]$ via \rev{$c^\ast \circ p_{n,k\ast}$} or equivalently \rev{$p_{n,k\ast}\circ c^\ast$}.

Let
\[
\hat{g}=g+  (\rev{c}^\ast\circ p_{n,k\ast})(\epsilon),
\]
where $+$ denotes the action. Since $\omega\circ p_{n,k}$ is null homotopic, $\hat{g}$ is a lifting of $\omega$ in Diagram \eqref{rliftprodia}. 
To investigate the difference of $\hat{g}\circ\rho$ and $g\circ \rho$, consider the homotopy commutative diagram
\begin{gather*}
\begin{aligned}
\xymatrix{
S^{2r-1} \ar[r]^{\mu \ \  \   \ \ \ \  \ \  \ }  \ar[d]^{\rho} &   S^{2r-1}\vee \bigvee\limits_{i=1}^{m} S^{2r-1}  \ar[d]^{\rho \vee\nabla} \ar[r]^{\ \ \ (g\circ \rho)\vee \big(\bigvee\limits_{i=1}^{m}p_{n,k\ast}(\epsilon)\big)} &   PV_{n,k}(\mathbb{C}) \vee \big(\bigvee\limits_{i=1}^{m}  PV_{n,k}(\mathbb{C})\big) \ar[d]^{\nabla}  \\
L^{2r-1}(m) \ar[r]^{\mu \ \ \ \ \ \ \ } \ar@/_1.7pc/[rr]_<<<<<<<<<<<{\ \ \hat{g}=g+  (\rev{c}^\ast\circ  p_{n,k\ast})(\epsilon)}           &  L^{2r-1}(m)\vee S^{2r-1} \ar[r]^{\ \ \ \ \  \ \ \hat{g}=g\vee p_{n,k\ast}(\epsilon)} & PV_{n,k}(\mathbb{C}),  
}
\end{aligned}
\label{liftcoactdiag}
\end{gather*}
where $\mu$ is the co-action map or iterated co-multiplication, $\nabla$ is the folding map, and the left square homotopy commutes as ${\rm deg}(\rho)=m$. 
It follows that
\[
\hat{g}\circ\rho-g\circ \rho=\nabla\circ  \big((g\circ \rho)\vee \big(\bigvee\limits_{i=1}^{m}p_{n,k\ast}(\epsilon)\big)\big)\circ\mu -g\circ \rho
=g\circ \rho+m p_{n,k\ast}(\epsilon)-g\circ \rho=p_{n,k\ast}(m\epsilon).
\]
By Lemma \ref{rmodifydegreelemma}, the induced equivariant map $\tilde{\hat{g}}$
is of degree
\[
{\rm deg}(\tilde{\hat{g}})={\rm deg}(\tilde{g})+{\rm deg}(m\epsilon)={\rm deg}(\tilde{g})+m\cdot t\cdot h_{n,k}(r).
\]
The proof of the proposition is completed.
\end{proof}

\begin{corollary}\label{0inftyequicoro}
Let $m\in \rev{\mathbb{Z}_{>0}}$ and $\rev{r\in \{n-k+1,\ldots,n\}}$.
The set of degrees of $\mathbb{Z}/m$-equivariant maps from $S^{2r-1}$ to $V_{n,k}(\mathbb{C})$ can only be an infinite set or empty.\hfill $\Box$
\end{corollary}

\begin{corollary}\label{h=1equicoro}
Let $(n,k,r)$ be a triple of positive integers such that $h_{n,k}(r)=1$ and $n-k+1\leq r \leq n$. Suppose that the set of degrees of $\mathbb{Z}/m$-equivariant maps from $S^{2r-1}$ to $V_{n,k}(\mathbb{C})$ is not empty. Then the set is
\[
\Big\{a~|~a\equiv \binom{n}{r}~{\rm mod}~m\Big\}.
\]
\end{corollary}
\begin{proof}
This follows immediately from Proposition \ref{rdegbarflemma} and Lemma \ref{neccondleema3}.
\end{proof}

\begin{proof}[Proof of Theorem \ref{mainthm}]
The theorem is a combination of Lemma \ref{degfhlemma} and Corollary \ref{h=1equicoro}.
\end{proof}

\section{Special case: $r=n-k+1$}\label{case1sec}
There is a canonical bundle diagram 
\begin{gather*}
\begin{aligned}
\xymatrix{
\mathbb{Z}/m \ar[r]  \ar@{^{(}->}[d]   & 
S^{2r-1} \ar@{=}[d] \ar[r]^<<<<{\rho} &
L^{2r-1}(m)\ar[d]^{\tau}\\
S^1 \ar[r]   &
S^{2r-1} \ar[r]^<<<<{h}  &
\mathbb{C}P^{r-1},
}
\end{aligned}
\label{lencpmapdia}
\end{gather*}
where $\tau$ is the map of orbit spaces. 
Let $r=n-k+1$.
The lifting problem
\begin{gather*}
\begin{aligned}
\xymatrix{
&& PV_{n,k}(\mathbb{C}) \ar[d]^{\omega}  \\
L^{2(n-k+1)-1}(m) \ar[r]^<<<{\tau }  &
\mathbb{C}P^{n-k} \ar@{-->}[ru]^{g}   \ar@{^{(}->}[r]^{i}
&  \mathbb{C}P^{\infty},
}
\end{aligned}
\label{n-k+1liftprodia}
\end{gather*}
admits a solution $g$ since the classifying map $\omega$ is $(2(n-k)+1)$-connected and $\mathbb{C}P^{n-k}$ has dimension $2(n-k)$. Choose a smooth map $g$. Then the previous diagrams imply a principal bundle diagram  
\begin{gather*}
\begin{aligned}
\xymatrix{
\mathbb{Z}/m \ar[r] \ar@{^{(}->}[d]   &  S^{2(n-k+1)-1} \ar[r]^<<<<{\rho} \ar@{=}[d]  &  L^{2(n-k+1)-1}(m) \ar[d]^{\tau}\\
S^1 \ar@{=}[d]    \ar[r] &S^{2(n-k+1)-1} \ar[r]^{h} \ar[d]^{\tilde{g}}  &  \mathbb{C}P^{n-k} \ar[d]^{g}\\
S^1              \ar[r]   &  V_{n,k}(\mathbb{C}) \ar[r]^{p_{n,k}}   &   PV_{n,k}(\mathbb{C}),
}
\end{aligned}
\label{n-k+1bundleladderdia}
\end{gather*}
providing an $S^1$-equivariant map 
\[
\tilde{g}: S^{2(n-k+1)-1}\larrow V_{n,k}(\mathbb{C}).
\]
\begin{lemma}\label{n-k+1exlemma}
\rev{The map $\tilde g$ satisfies}
\[
{\rm deg}(\tilde{g})= \binom{n}{n-k+1}.
\]
\end{lemma}
\begin{proof}
Consider the morphism of Serre spectral sequences induced by the fibration diagram 
\begin{gather*}
\begin{aligned}
\xymatrix{
S^{2(n-k+1)-1} \ar[r]^{h} \ar[d]^{\tilde{g}}  &  \mathbb{C}P^{n-k} \ar[d]^{g} \ar[r]^{\rev{i}} & \mathbb{C}P^{\infty} \ar@{=}[d] \\
  V_{n,k}(\mathbb{C}) \ar[r]^{p_{n,k}}   &   PV_{n,k}(\mathbb{C}) \ar[r]^{\omega} & \mathbb{C}P^{\infty}.
}
\end{aligned}
\label{n-k+1cppvnkmapdia}
\end{gather*}
It is known that the transgressions satisfy
\[
\tau(s_{2(n-k+1)-1})=\omega^{n-k+1}, \ \  \tau(x_{2(n-k+1)-1})=\binom{n}{n-k+1} \omega^{n-k+1},
\]
for the two fibrations respectively, \rev{where $s_{2(n-k+1)-1}$ denotes the generator in the spectral sequence of the Hopf fibration, while $x_{2(n-k+1)-1}$ denotes the corresponding generator in the spectral sequence for $V_{n,k}(\mathbb{C})\to PV_{n,k}(\mathbb{C})$}. Hence, by the naturality of the transgression, the lemma follows.
\end{proof}
\begin{theorem}\label{n-k+1degfthm}
\rev{The equivariant degree set is}
 \[
\rev{D_m}(S^{2(n-k+1)-1}, V_{n,k}(\mathbb{C}))=\Big\{a~|~a\equiv \binom{n}{n-k+1}~{\rm mod}~m\Big\}.
\]
\end{theorem} 
\begin{proof}
By Corollary \ref{1-3hnkrcorollary}, the Hurewicz index \rev{$h_{n,k}(n-k+1)=1$}. \rev{The construction preceding Lemma~\ref{n-k+1exlemma} gives an $S^1$-equivariant, hence $\mathbb{Z}/m$-equivariant, map, so the equivariant degree set is nonempty.} Then the theorem follows from Corollary \ref{h=1equicoro}.
\end{proof}

\section{Special case: $r=n-k+2$}\label{case2sec}
We may apply the material of \cite[Section 5]{AGMP99} to investigate equivariant maps from a geometric point of view. Consider the lifting diagram
\begin{gather*}
\begin{aligned}
\xymatrix{
X  \ar@{-->}@/^1.3pc/[drr]^{\tilde{h}}   \ar@/_1.3pc/[ddr]^{f_{\lambda}} \ar@{-->}[dr]^{h} \\ 
     &PV_{n,k}(\mathbb{C}) \ar[r]^{f} \ar[d]^{\omega}   & B\mathbb{U}(n-k)\ar[d]^{Bi}\\
     & \mathbb{CP}^{\infty} \ar[r]^{f_0}   & B\mathbb{U}(n),
}
\end{aligned}
\label{liftnlinediag}
\end{gather*}
where the lower right square is the pullback in Diagram \eqref{pvnkdefdiag}, and $f_{\lambda}$ classifies a given complex line bundle $\lambda$ over a space $X$. Hence, $f_0\circ f_{\lambda}$ classifies the bundle $n\lambda$. Then there exists a lifting $h$ if and only if there exists a lifting $\tilde{h}$. The latter is equivalent to saying that the bundle $n\lambda$ has a trivial subbundle of complex dimension $k$. 

When $X=L^{2r-1}(m)$, we can relate the triviality of a subbundle of $n\lambda$ to the existence of certain equivariant map.
\begin{lemma}\label{trivilsubbunlemma}
\rev{Let $\lambda$ be the complex line bundle over $L^{2r-1}(m)$ associated to the canonical principal bundle $S^{2r-1}\stackrel{\rho}{\larrow} L^{2r-1}(m)$. The Whitney sum $n\lambda$ has a trivial subbundle of complex dimension $k$}
if and only if there exists a $\mathbb{Z}/m$-equivariant map
\[
f: S^{2r-1}\larrow V_{n,k}(\mathbb{C}).
\] 
\end{lemma}
\begin{proof}
By the previous discussion, the bundle $n\lambda$ has a trivial subbundle of complex dimension $k$ if and only if the classifying map $f_\lambda=\omega: L^{2r-1}(m)\larrow\mathbb{C}P^{\infty}$ of the line bundle $\lambda$ can be lifted to a map $L^{2r-1}(m)\larrow PV_{n,k}(\mathbb{C})$, or equivalently, the lifting problem \eqref{rliftprodia} has a solution. The latter is further equivalent to the existence of a $\mathbb{Z}/m$-equivariant map $S^{2r-1}\larrow V_{n,k}(\mathbb{C})$ by \eqref{assuequigeq} and Lemma \ref{fslemma}.
\end{proof}

\begin{example}\label{n-k+1ex}
In Section \ref{case1sec}, we have constructed an $S^1$-equivariant map 
\[
\tilde{g}: S^{2(n-k+1)-1}\larrow V_{n,k}(\mathbb{C}).
\]
\rev{The map $\tilde{g}$ constructed there is $S^1$-equivariant and does not depend on a choice of finite subgroup; hence, after restricting the $S^1$-action, it is $\mathbb{Z}/m$-equivariant for every $m\in \mathbb{Z}_{>0}$.} Let $\lambda$ be the complex line bundle associated to the circle bundle $S^{2(n-k+1)-1}\stackrel{\rho}{\larrow} L^{2(n-k+1)-1}(m)$. By Lemma \ref{trivilsubbunlemma}, it follows that the bundle $n\lambda$ has a trivial subbundle of complex dimension $k$, or equivalently, the classifying map $f_0\circ f_\lambda=f_0\circ \omega$ of $n\lambda$ can be lifted from $B\mathbb{U}(n)$ to $B\mathbb{U}(n-k)$.
\end{example}

\begin{example}\label{n-k+3ex}
Let $\lambda$ be the complex line bundle associated to the circle bundle $S^{2(n-k+3)-1}\stackrel{\rho}{\larrow} L^{2(n-k+3)-1}(m)$. Suppose that $\nu_2\binom{n}{n-k+1}=\nu_2(m)>0$, and $k$ is odd if $n$ is even. 
In \cite[Theorem 1.3]{AGMP99}, Astey-Gitler-Micha-Pastor showed that the bundle $n\lambda$ does not admit a trivial subbundle of complex dimension $k$. By Lemma \ref{trivilsubbunlemma}, this is equivalent to saying that there does not exist a $\mathbb{Z}/m$-equivariant map 
\[
S^{2(n-k+3)-1}\larrow V_{n,k}(\mathbb{C}).
\]
\end{example}

Now we specify to the special case when $r=n-k+2$. 

\begin{theorem}\label{n-k+2degfthm}
Suppose $n-k\geq 2$ is even, and 
\[
\binom{n}{1}\equiv \binom{n}{k-1}\equiv \binom{n}{k}\equiv 0~{\rm mod}~m.
\]
Then there is a $\mathbb{Z}/m$-equivariant map $f:S^{2(n-k+2)-1}\larrow V_{n,k}(\mathbb{C})$.
Additionally, when the above conditions hold, 
 \[
\rev{D_m}(S^{2(n-k+2)-1}, V_{n,k}(\mathbb{C}))=\Big\{a~|~a\equiv \binom{n}{n-k+2}~{\rm mod}~m\Big\}.
\]
\end{theorem}
\begin{proof}
The Chern class of $n\lambda$ over $L^{2(n-k+2)-1}(m)$ is 
\[
c(n\lambda)=(1+\rev{\alpha})^n,
\]
where the generator \rev{$\alpha \in H^{2}(L^{2(n-k+2)-1}(m);\mathbb{Z})\cong \mathbb{Z}/m$} is the Euler class of the line bundle $\lambda$. Then the assumption $\binom{n}{1}\equiv \binom{n}{k-1}\equiv \binom{n}{k}\equiv 0~{\rm mod}~m$ is  equivalent to $c_1(n\lambda)=c_{n-k}(n\lambda)=c_{n-k+1}(n\lambda)=0$.

Consider the $S^1$-equivariant map $\tilde{g}: S^{2(n-k+2)-1}\larrow V_{n,k-1}(\mathbb{C})$ constructed in Section \ref{case1sec}. By Example \ref{n-k+1ex}, the existence of $\tilde{g}$ implies that the classifying map $f_0\circ f_\lambda=f_0\circ \omega$ of $n\lambda$ can be lifted to $B\mathbb{U}(n-k+1)$. It can be further lifted to \rev{$B\mathbb{SU}(n-k+1)$} since $c_1(n\lambda)=0$. 
Then as $c_{n-k}(n\lambda)=c_{n-k+1}(n\lambda)=0$, the classifying map $f_0\circ f_\lambda=f_0\circ \omega$ can be lifted further to \rev{$B\mathbb{SU}(n-k)$} by \cite[Corollary 5.12]{AGMP99}, or equivalently, the bundle $n\lambda$ over $L^{2(n-k+2)-1}(m)$ has a trivial subbundle of complex dimension $k$. By Lemma \ref{trivilsubbunlemma}, this is equivalent to the existence of a $\mathbb{Z}/m$-equivariant map $S^{2(n-k+2)-1}\larrow V_{n,k}(\mathbb{C})$.

The second statement follows from Corollary \ref{h=1equicoro} with the fact that when $n-k$ is even \rev{$h_{n,k}(n-k+2)=1$} by Corollary \ref{1-3hnkrcorollary}.
\end{proof}

We conclude this paper with a problem, taking the following proposition as a starting point.
\begin{proposition}\label{noncan-prop}
	Let $(n,k,r)$ be a triple of positive integers such that $n-k+1\leq r\leq n$ and 
	\[
	M_{r-(n-k)}~|~r.
	\]
	If there exists an isomorphism of complex vector bundles over $L^{2r-1}(m)$
	\[
	n\lambda\cong E\oplus \epsilon^{k},
	\]
	where $\epsilon^{k}$ denotes the trivial bundle of complex dimension $k$, then \rev{there is an infinite sequence of decompositions of this form, i.e., there exist complex $(n-k)$-bundles $E_i$ ($i\in \mathbb{N}$) over $L^{2r-1}(m)$ such that}
	\[
	n\lambda\cong E_i\oplus \epsilon^{k}~\text{for any $i$}.
	\]
\end{proposition}
\begin{proof}
By Lemma \ref{trivilsubbunlemma}, the assumption $n\lambda\cong E\oplus \epsilon^{k}$ implies that there exists a $\mathbb{Z}/m$-equivariant map $S^{2r-1}\larrow V_{n,k}(\mathbb{C})$. Then Theorem \ref{mainthm} implies that there exist infinitely many $\mathbb{Z}/m$-equivariant maps $f_i : S^{2r-1} \larrow V_{n,k}(\mathbb{C})$ whose equivariant degrees are pairwise distinct. By Lemma \ref{trivilsubbunlemma} again, each $f_i$ determines a decomposition $n\lambda\cong E_i\oplus \epsilon^{k}$ for some bundle $E_i$.
\end{proof}

Under the assumption of Proposition \ref{noncan-prop}, we obtain infinitely many classifying maps
\[
g_i: L^{2r-1}(m)\larrow B\mathbb{U}(n-k)
\]
classifying the bundles $E_i$, \rev{arising from equivariant maps whose degrees are pairwise distinct}. However, since $L^{2r-1}(m)$ is rationally homotopy equivalent to $S^{2r-1}$ and $\pi_{2r-1}(B\mathbb{U}(n-k))\cong \pi_{2r-2}(\mathbb{U}(n-k))$ is a torsion group, the set of homotopy classes $[L^{2r-1}(m),B\mathbb{U}(n-k)]$ is finite. This raises the following problem.
\begin{problem}\label{prob}
	For the bundles $E_i$ ($i\in \mathbb{N}$) constructed in Proposition \ref{noncan-prop},
\begin{itemize}
	\item[(1).] For which pairs $(i,j)$ does $E_i\cong E_j$ hold?
	\item[(2).] Does there exist a complex bundle $F$ over $L^{2r-1}(m)$ of rank $(n-k)$ such that $F\not\cong E_i$ for any $i$?
\end{itemize}
	\end{problem}




\begin{thebibliography}{10}
\bibitem{AW65} J. F. Adams and G. Walker, {\em On complex Stiefel manifolds}, Math. Proc. Camb. Phil. Soc. \textbf{61} (1965), 81-103.

\bibitem{AT60} M. F. Atiyah and J. A. Todd, {\em On complex Stiefel manifolds}, Math. Proc. Camb. Phil. Soc. \textbf{56} (1960), 342-353.

\bibitem{AGMP99} L. Astey, S. Gitler, E. Micha and G. Pastor. {\em Cohomology of complex projective Stiefel manifolds}, Canad. J. Math. \textbf{51} (1999), no. 5, 897-914. 

\bibitem{Borel53} A. Borel, {\em Sur la cohomologie des espaces fibres principaux et des espaces homogenes de groupes de Lie compacts}, Ann of Math. \textbf{57} (1953), 115-207. 


\bibitem{Brouwer11} L. E. J. Brouwer, {\em Uber Abbildung von Mannigfaltigkeiten}, Math. Ann. \textbf{71} (1911), 97-115.

\bibitem{Ch} C. Chevalley, {\em The Betti numbers of the exceptional Lie groups},
Proceedings of the International Congress of Mathematicians, Cambridge,
Mass., 1950, Providence, AMS, vol. 2 (1952), 21-24.

\bibitem{Duan20} H. Duan, The Cohomology of Projective Unitary Groups.~Proc.
Steklov Inst. Math.~326, 157--176 (2024).

\bibitem{DL17} H. Duan and X. Lin, {\em Topology of unitary groups and the prime orders of binomial coefficients}, Sci. China Math. \textbf{60} (2017), no. 9, 1543-1548. 

\bibitem{DW03} H. Duan and S. C. Wang, {\em The degrees of maps between manifolds}, Math. Z. \textbf{244} (2003), 67-89.

\bibitem{DW04} H. Duan and S. C. Wang, Non-zero degree maps between $2n$-manifolds, \emph{Acta Math. Sin. (Engl. Ser.)} \textbf{20} (2004), 1-14.

\bibitem{DW20} H. Duan and S. Wu, {\em The multi-degree of covering on Lie groups}, preprint, 2020.


\bibitem{James58} I. M. James, {\em Cross-sections of Stiefel manifolds}, Proc. London Math. Soc. (3) \textbf{8} (1958), 536-547.


\bibitem{McCleary01} J. McCleary, {\em A user's guide to spectral sequences}, Cambridge Studies in Advanced Mathematics \textbf{58}, Cambridge University Press, (2001). 

\bibitem{Ruiz69} \rev{C. A. Ruiz, {\em The cohomology of the complex projective Stiefel manifold}, Trans. Amer. Math. Soc. \textbf{146} (1969), 541-547.}


\bibitem{Wall62} C. T. C. Wall, {\em Classification of $(n-1)$-connected $2n$-manifolds}, Ann. of Math. \textbf{75} (1962), 163-189.

\bibitem{Wan02} S. C. Wang, Non-zero degree maps between 3-manifolds, \emph{Proceedings of the ICM}, Beijing 2002, vol. 2, 457-470.


\end{thebibliography}
\end{document}